\documentclass[reqno]{amsart}
\usepackage{amssymb, latexsym}
\usepackage{hyperref}

\newcommand{\N}{{\mathbb{N}}}

\newcommand{\C}{{\mathbb{C}}}
\newcommand{\R}{{\mathbb{R}}}

\let\Im=\undefined\DeclareMathOperator*{\Im}{Im}

\DeclareMathOperator*{\dist}{dist}
\DeclareMathOperator*{\Id}{Id}

\newcommand\qtq[1]{{\quad\text{#1}\quad}}

\newcommand{\eps}{{\varepsilon}}
\newcommand{\tE}{\mathcal{\tilde E}}
\newcommand{\rstrut}{{\vrule width 0ex depth 0ex height 1.2ex}}

\theoremstyle{plain}
\newtheorem{theorem}{Theorem}
\newtheorem{proposition}[theorem]{Proposition}
\newtheorem{lemma}[theorem]{Lemma}
\newtheorem{corollary}[theorem]{Corollary}
\theoremstyle{definition}

\newtheorem{definition}[theorem]{Definition}
\theoremstyle{remark}

\newtheorem*{remarks}{Remarks}

\newenvironment{CI}{\begin{list}{{\ $\bullet$\ }}{%
\setlength{\topsep}{0mm}\setlength{\parsep}{0mm}\setlength{\itemsep}{0mm}%
\setlength{\labelwidth}{0mm}\setlength{\itemindent}{0mm}\setlength{\leftmargin}{0mm}%
\setlength{\labelsep}{0mm} }}{\end{list}}

\numberwithin{equation}{section}
\numberwithin{theorem}{section}

\begin{document}

\title[Blowup on smooth hypersurfaces for NLW]{Smooth solutions to the nonlinear wave equation can blow up on Cantor sets}
\author{Rowan Killip}
\address{Department of Mathematics\\UCLA\\Los Angeles, CA 90095}
\author{Monica Vi\c{s}an}
\address{Department of Mathematics\\UCLA\\Los Angeles, CA 90095}

\begin{abstract}
We construct $C^\infty$ solutions to the one-dimensional nonlinear wave equation
$$
u_{tt} - u_{xx} - \tfrac{2(p+2)}{p^2} |u|^p u=0 \quad \text{with} \quad  p>0
$$
that blow up on any prescribed uniformly space-like $C^\infty$ hypersurface.  As a corollary, we show that smooth solutions
can blow up (at the first instant) on an arbitrary compact set.

We also construct solutions that blow up on general space-like $C^k$ hypersurfaces, but only when $4/p$ is not an integer and $k > (3p+4)/p$.
\end{abstract}

\maketitle

\section{Introduction}

In this paper we consider the one-dimensional focusing nonlinear wave equation
\begin{equation}\label{nlw}
u_{tt} - u_{xx} - \tfrac{2(p+2)}{p^2} |u|^p u=0
\end{equation}
for general powers $p>0$.  The coefficient of the nonlinearity can be changed via the rescaling $u_\lambda(t,x):=u(\lambda t,\lambda x)$.  However, this particular
choice simplifies many of the formulae that follow.

Even for smooth initial data, solutions to \eqref{nlw} can blow up in finite time.  The simplest example of this is
\begin{equation}\label{simple model}
u(t,x) = (T-t)^{-2/p} \qtq{defined for} t<T \qtq{and} x\in\R.
\end{equation}
(Note the benefit of choosing the coefficient $2(p+2)/p^2$ in \eqref{nlw}.)

The equation \eqref{nlw} is relativistically invariant.  This blurs the distinction between space and time, as well as implying finite speed of propagation.
As a result, it is natural to consider the maximal space-time region on which the solution is defined.  Moreover, this region must take the form
$$
\{(t,x) : \tilde\sigma(x) < t < \sigma(x) \},
$$
where either $\sigma(x)\equiv \infty$ or $\sigma$ is a $1$-Lipschitz function, that is,
$$
\sup_{x,y} \frac{|\sigma(x)-\sigma(y)|}{|x-y|} \leq 1,
$$
and either $\tilde \sigma(x)\equiv-\infty$ or $\tilde\sigma$ is a $1$-Lipschitz function.  (The number $1$ represents the speed of light in our model.)  A point $t=\sigma(x)$ on the (forward)
blowup surface is called \emph{non-characteristic} if
$$
 \limsup_{y\to x} \frac{|\sigma(x)-\sigma(y)|}{|x-y|} < 1;
$$
otherwise, it is called \emph{characteristic}.  The surface $\{t=\sigma(x)\}$ is called \emph{space-like} if it is everywhere non-characteristic.

In this paper we discuss the question of which surfaces can occur as blowup surfaces.  In the papers \cite{CF1,CF2}, Caffarelli and Friedman showed that
$\sigma$ must be $C^1$ at non-characteristic points for a closely related model (they also obtain conditional results in higher dimensions).  More recently,
Merle and Zaag considered precisely the model \eqref{nlw} in a series of papers; see \cite{MerleZaag1,MerleZaag2} and the references therein.
They have shown that the characteristic points form a discrete set and that $\sigma$ is $C^1$ away from these points.  (The also obtain a wealth of information
about how the solution behaves near the blowup surface.)  In \cite{Nouaili}, Nouaili extended their methods to show that away from the characteristic points,
$\sigma$ is $C^{1,\alpha}$ for some (unspecified) H\"older exponent $\alpha>0$.

In the other direction, one may seek to show that certain rich classes of functions $\sigma$ do indeed arise as blowup surfaces.  The only works we know of that
address this type of question are the papers \cite{KL1,KL2} of Kichenassamy and Littman.  These authors consider arbitrary spatial dimension with $p\in\N$ and show
that any analytic space-like surface can occur as a blowup surface.  In his lecture notes on the subject, Alinhac explicitly asks about the possibility of extending the work of Kichenassamy and Littman to the $C^\infty$ case; see \cite[p. 9]{Alinhac}.  In this paper, we work only in one dimension and show that general $C^\infty$ space-like surfaces can indeed occur as blowup surfaces.  Our methods apply also in the case of finite regularity; however, for technical reasons, the argument becomes messier if $4/p$ is an integer, so we are omitting that case.

\begin{theorem}[Regular solutions that blow up on a space-like $C^k$ hypersurface]\label{T:main}\leavevmode\\
Suppose either {\rm(i)} $p>0$ and $k=\infty$, or {\rm(ii)} $\frac4p\not\in \N$ and $3+\frac4p<k\in \N\cup\{\infty\}$.  Let $\Sigma=\{ t=\sigma(x)\}$
be a $C^k$ uniformly space-like hypersurface {\rm(}i.e., $\sup_x |\sigma'(x)| < 1${\rm)}.  Then there exist a nearby space-like hypersurface $\Sigma_0:=\{t=\sigma_0(x)\}$
with $\sigma_0(x)<\sigma(x)$ and a solution $u\in C^{m}_{t,x}(\{\sigma_0(x)\leq t<\sigma(x)\})$ to \eqref{nlw}, with $m=\lfloor k-3-4/p\rfloor$, that blows up on $\Sigma$ in the sense that
\begin{align}\label{blowup}
\lim_{t\nearrow \sigma(x)} \bigl(\sigma(x) -t\bigr)^{\frac2p}u(t,x) = \bigl( 1-|\sigma'(x)|^2\bigr)^{\frac1p}
\end{align}
for all $x\in\R$.
\end{theorem}

\begin{remarks} 1. The floor notation, $\lfloor x\rfloor$, denotes the greatest integer less than or equal to $x$.  Similarly, the ceiling $\lceil x \rceil$ is the
least integer greater than or equal to $x$.

2. The blowup rate in \eqref{blowup} is Lorentz invariant: to this degree of accuracy, $u(t,x)$ agrees with the Lorentz boost of the solution \eqref{simple model} that blows up
on the tangent line to $\Sigma$ passing through the point $t=\sigma(x)$.

3. The space-like hypersurface $\{t=\sigma_0(x)\}$ may follow the blowup surface $\{t=\sigma(x)\}$ so closely that our solution $u$ is not defined globally in space
at any one particular time.  (This issue also appears in results of \cite{KL2}.)

4. The condition $k>3+\frac4p$ arises as a natural requirement for our method.  This is discussed at some length in Section~\ref{S:construction}.
\end{remarks}

Our initial interest in treating the case of smooth (as opposed to analytic) blowup surfaces was to address the question alluded to in the title:
Can solutions to nonlinear wave equations blow up on Cantor-like sets?  By this we mean that at the time when the first singularity occurs, can this singularity arise simultaneously
across a Cantor-like set?  Note that if $\sigma$ is an analytic function, then $\{x : \sigma(x)=T\}$ must be discrete (or all of $\R$); however, given an arbitrary compact set $E$,
it is not difficult to construct a smooth function for which $E$ is a level set.  In this way, Theorem~\ref{T:main} allows us to construct a solution to \eqref{nlw} that blows up on an arbitrary compact set.  Moreover, we are able to control the time of existence sufficiently well that we can guarantee that the solution admits global (in space) Cauchy data (cf. Remark~3 above).
The precise statement is as follows:

\begin{corollary}[Blowup on arbitrary compact sets]\label{C:Cantor}
Fix $p>0$.  Given a compact set $E\subset \R$, there exist $C^\infty(\R)$ initial data $(u(0), u_t(0))$ and a smooth solution to \eqref{nlw} with this initial data that blows up at a time $0<T<\infty$ precisely on the set~$E$, which is to say,
\begin{CI}
\item $u(t,x)$ converges to a smooth function of $x \in \R\setminus E$ when $t\to T$ and
\item $\lim_{t\to T} u(t,x) =\infty$ for all  $x\in E$.
\end{CI}
\end{corollary}

The proof of Theorem~\ref{T:main} proceeds by constructing a parametrix $\tilde u$ for the desired solution and then using contraction mapping arguments
to correct it to an actual solution $u$ to \eqref{nlw}.  Due to the singularity of the solution we are seeking, the equation obeyed by the correction term
$u-\tilde u$ is of Fuchsian type, which is a PDE analogue of an ODE having a regular singular point.  This has two consequences.  First, the parametrix
needs to be constructed to very high order.  This is done in Section~\ref{S:parametrix}.  Second, the contraction mapping argument revolves around a singular
variant of the linear wave equation.  The basic estimates for this equation are derived in Section~\ref{S:estimates} and then used to construct the solution in
Section~\ref{S:construction}.

Working in one spatial dimension gives rise to a number of simplifications, both mathematical and expository.  We make particular use of the Minkowski-space analogue
of the Riemann mapping theorem (Proposition~\ref{P:conformal exists} below).  This allows us to choose coordinates $(s,y)$ for the space-time region $\{t<\sigma(x)\}$ that straighten out the boundary while only deforming the d'Alembertian, $\partial_t^2-\partial_x^2$, by a conformal factor.  That this is only possible in $\R^{1+1}$ is a famous discovery of Liouville;
specifically, the only (Euclidean or Minkowski) conformal mappings in higher dimensions are M\"obius transformations (see \cite[\S15]{ModernGeom1}).  As the notion of
a conformal mapping relative to a pseudo-Riemannian metric is discussed in relatively few differential geometry textbooks (\cite{ModernGeom1} is one exception), we develop the facts
we require from scratch in Section~\ref{S:conformal}.

While M\"obius transformations in Minkowski space have played a central role in several fundamental developments in the theory of nonlinear wave equations (spectacularly so at the hands of C.~Morawetz or D.~Christodoulou, for example), we have not seen the greater generality available in the $\R^{1+1}$ setting used previously to aid in the analysis of this
type of equation.

\subsection*{Acknowledgements}
We are grateful to Mario Bonk for valuable discussions about conformal mappings in Minkowski space.

The first author was partially supported by NSF grant DMS-1001531.  The second author was partially supported by the Sloan Foundation and NSF grant DMS-0901166.
This work was completed while the second author was a Harrington Faculty Fellow at the University of Texas at Austin.

%
%
%
%

\section{Existence and properties of conformal maps} \label{S:conformal}

In this section, we use a conformal mapping (relative to the Minkowski metric) to reduce the problem of constructing solutions to \eqref{nlw} that
blow up on a (sufficiently smooth) space-like hypersurface to that of constructing solutions that blow up on the spatial axis; these latter solutions must solve a suitably modified
nonlinear wave equation.

\begin{definition}[Conformal mapping]\label{D:conformal}
A $C^1$ diffeomorphism $\Phi:(s,y)\mapsto(t,x)$ between two open sets in Minkowski space $\R^{1+1}$ is called \emph{conformal} if there exists a \emph{conformal factor} $\lambda(s,y)>0$ such that
\begin{equation}\label{E:Conf}
\begin{aligned}
\begin{bmatrix} \tfrac{\partial t}{\partial s} \, & \tfrac{\partial t}{\partial y}\\[1ex]  \tfrac{\partial x}{\partial s} \, & \tfrac{\partial x}{\partial y}\end{bmatrix}
\begin{bmatrix} -1 \, & 0\\[1ex]  0 \, & 1\end{bmatrix}
\begin{bmatrix} \tfrac{\partial t}{\partial s} \, & \tfrac{\partial x}{\partial s}\\[1ex] \tfrac{\partial t}{\partial y}  \, & \tfrac{\partial x}{\partial y}\end{bmatrix}
= \lambda(s,y)\begin{bmatrix} -1 \, & 0\\[1ex]  0 \, & 1\end{bmatrix}.
\end{aligned}
\end{equation}
Equivalently, the pull back of the Minkowski metric via $\Phi$ differs from the original Minkowski metric by the scalar factor $\lambda(s,y)>0$.  The mapping $\Phi$ is a (Minkowski)
\emph{isometry} at those points where $\lambda(s,y)=1$.
\end{definition}

From \eqref{E:Conf} it is clear that the null (or light-like) directions are unchanged by a conformal mapping.  As $\lambda>0$, the notions of space- and time-like are preserved.  Thus
there are four classes of conformal mappings depending on whether the spatial and/or temporal orientations are preserved or inverted.  When both are preserved, the mapping is called
\emph{proper} (cf. \cite[\S6.2]{ModernGeom1}).

It is easy to derive the Minkowski analogue of the Cauchy--Riemann equations from \eqref{E:Conf}.  In the case of proper mappings, they read
$$
\tfrac{\partial t}{\partial s}=\tfrac{\partial x}{\partial y} >0 \quad\text{and}\quad \tfrac{\partial t}{\partial y}=\tfrac{\partial x}{\partial s}.
$$
In particular, $(t,x)$ are wave coordinates, that is, $t_{ss}-t_{yy}=0$ and $x_{ss}-x_{yy}=0$ (in the sense of distributions).
Using the general solution of the one-dimensional wave equation, we deduce that all proper conformal mappings of Minkowski space take the form
\begin{equation}\label{E:fg}
t= \tfrac12\bigl[ f(y+s)-g(y-s) \bigr] \quad\text{and}\quad x= \tfrac12\bigl[ f(y+ s) + g(y-s) \bigr],
\end{equation}
where $f$ and $g$ are orientation-preserving $C^1$ diffeomorphisms of $\R$.  (The remaining three types of conformal mappings correspond to allowing $f$ and/or $g$ to be
orientation-reversing.) A more geometrical derivation of this fact (which we learned from Mario Bonk) is as follows:  As conformal mappings preserve the null directions,
the grid of light rays in $(t,x)$ coordinates must be mapped onto the corresponding grid in $(s,y)$.  Thus, a conformal mapping corresponds to an (arbitrary and independent)
change of variables in each of the null directions: $(x+t,x-t)= (f(y+s), g(y-s))$, which is, of course, just another way of writing \eqref{E:fg}.

\begin{proposition}[Existence of conformal maps]\label{P:conformal exists}
Let $2\leq k\in \N\cup\{\infty\}$ and consider a $C^k$ uniformly space-like hypersurface $\Sigma=\{t=\sigma(x)\}$.  Then there exists a $C^k$ conformal diffeomorphism from
$\{s\geq 0, \, y\in\R\}$ onto $\{t \leq \sigma(x)\}$, which is an isometry between the boundaries.  Moreover, the conformal factor $\lambda(s,y)$ is $C_{s,y}^{k-1}$.
\end{proposition}

\begin{proof}
The natural conformal mapping $\Phi:(s,y)\mapsto (t,x)$ is not proper.  We seek instead one for which
$(s,y)\mapsto\Phi(-s,y)$ is proper, which, from the preceding discussion, must take the form
\begin{align}\label{flipped fg}
x+t &= f(y-s)  \quad\text{and}\quad  x-t=g(y+s)
\end{align}
with increasing functions $f,g \in C^k(\R)$.

As $\Sigma$ is uniformly space-like, $|\sigma'(x)|<1$ uniformly for $x\in \R$.  Thus, by the implicit function theorem, there exists a $C^k$ function $h:\R\to \R$ so that
$$
t=\sigma(x)\quad  \text{if and only if} \quad x+t = h(x-t).
$$
Moreover, $\eps^2 < h'(z)= \frac{ 1+\sigma'(z/2+ h(z)/2)}{ 1-\sigma'(z/2+ h(z)/2)} < \eps^{-2}$ for some $\eps >0$.

The condition that the conformal mapping sends $\{s=0, \, y\in \R\}$ onto $\Sigma$ reads
$$
f(y)=h(g(y)) \qquad \text{for all $y\in\R$,}
$$
while the condition that it is an isometry on the boundary is
$$
f'(y) g'(y) =1  \qquad \text{for all $y\in\R$}.
$$
Therefore, $g$ must solve the following ODE:
$$
g'(y) = \bigl[ h'(g(y)) \bigr]^{-1/2} \quad \text{with, say,} \quad g(0)=0.
$$
(We take the positive square-root.)  By Picard's existence and uniqueness theorem, this equation admits a unique global solution $g\in C^k$ for $h\in C^k$ with $k\geq 2$.
Thus $f=h\circ g \in C^k$, as well.  Moreover  $g'(y)>\eps$ and $f'(y)> \eps$, so we are guaranteed that both are diffeomorphisms onto the whole of $\R$.

A computation shows that the conformal factor is given explicitly by
$$
\lambda(s,y):=f'(y-s)g'(y+s)>0,
$$
which is clearly $C^{k-1}_{s,y}$.  It is now an easy matter to verify that the conformal mapping we constructed satisfies the desired properties.  This completes the proof of the proposition.
\end{proof}

A direct consequence of Proposition~\ref{P:conformal exists} is the following

\begin{corollary}[Reduction to blowup on lines]\label{C:line reduction}
For any $2\leq k\in \N\cup\{\infty\}$, there is a one-to-one correspondence between solutions to \eqref{nlw} that blowup on a $C^k$ uniformly space-like hypersurface in the sense of
\eqref{blowup} and solutions to
\begin{align}\label{mnlw}
v_{ss}   - v_{yy}  - \tfrac{2(p+2)}{p^2} \lambda(s,y) |v|^p v = 0
\end{align}
that blow up on $\{s=0, \, y\in \R\}$ in the sense that
\begin{equation}\label{sy blowup}
\lim_{s\searrow 0}s^{\frac2p} v(s,y) =1
\end{equation}
for each $y\in \R$.  Here, $\lambda(s,y)$ denotes the conformal factor from Proposition~\ref{P:conformal exists}.
\end{corollary}

\begin{proof}
Let $\Phi:(s,y)\mapsto (t,x)$ be the conformal mapping from Proposition~\ref{P:conformal exists}.  A straightforward computation shows that
$$
\partial_s^2 - \partial_y^2 = \lambda(s,y) (\partial_t^2-\partial_x^2),
$$
which means that $v(s,y)=u\circ \Phi(s,y)$ solves \eqref{mnlw} if and only if $u(t,x)$ solves \eqref{nlw}.

To prove the correspondence between the blowup behavior of $v$ on the line $\{s=0, \, y\in \R\}$ and that of $u$ on the hypersurface $\Sigma$, we merely have to show that
\begin{align}\label{coresp}
\lim_{t\nearrow \sigma(x)} \frac{\sigma(x)-t}s = - \sqrt{ 1-|\sigma'(x)|^2 } = \lim_{s\searrow 0} \frac{\sigma(x)-t}s ,
\end{align}
where $(t,x)=\Phi(s,y)$.  In the former limit, $x$ is held fixed; in the latter, $y$ is fixed.  To verify these relations, we need to compute the derivative of $\Phi$ (and its inverse)
at the boundary.

Recall from the proof of Proposition~\ref{P:conformal exists} that our conformal mapping takes the form
$$
t= \tfrac12\bigl[ f(y-s)-g(y+s) \bigr] \quad\text{and}\quad x= \tfrac12\bigl[ f(y-s) + g(y+s) \bigr],
$$
with $f$ and $g$ increasing and obeying
$$
f'(y) = \frac{1}{g'(y)} = \bigl[ h'(g(y)) \bigr]^{1/2} = \bigl[ \tfrac{1+\sigma'(x)}{1-\sigma'(x)} \bigr]^{1/2}
$$
when $x=\tfrac12[ f(y) + g(y)]$, that is, when $s=0$.  Therefore,
\begin{align*}
\begin{bmatrix} \tfrac{\partial t}{\partial s} \, & \tfrac{\partial t}{\partial y}\\[1ex]  \tfrac{\partial x}{\partial s} \, & \tfrac{\partial x}{\partial y}\end{bmatrix}(0,y)
 &= \frac12 \begin{bmatrix} -f'(y)-g'(y) & f'(y)-g'(y) \\[1ex]  -f'(y)+g'(y) & f'(y)+g'(y) \end{bmatrix} \\
 &= \frac{1}{\sqrt{1-|\sigma'(x)|^2}} \begin{bmatrix} -1 & \sigma'(x) \\[1ex]  -\sigma'(x) & 1 \end{bmatrix}
\end{align*}
where, as before,  $(0,y)$ and $(\sigma(x),x)=\Phi(0,y)$ are corresponding points on the boundaries.  The left column of this matrix justifies the second identity
in \eqref{coresp}, while the top left entry of its inverse explains the first identity.  Note that this matrix has determinant $-1$, as can be predicted from
the knowledge that $\lambda=1$ at the boundary and $\Phi$ is orientation-reversing in the first variable and orientation-preserving in the second.
\end{proof}

%
%
%
%

\section{Building the parametrix} \label{S:parametrix}
In this section, we further analyze the modified wave equation
\begin{align}\label{modified nlw}
v_{ss}   - v_{yy}  - \tfrac{2(p+2)}{p^2} \lambda(s,y) |v|^p v =0,
\end{align}
where $\lambda$ denotes the conformal factor introduced in the previous section.  We remind the reader that the requirement that the conformal
mapping be an isometry on the boundary translates to $\lambda(0,y) =1$ for all $y\in \R$.  Our goal is to find a good parametrix for \eqref{modified nlw}
that blows up on the line $\{s=0,\, y\in\R\}$.

Let us start by recalling a very simple version of the well-known extension theorems of Whitney and Borel (cf. Theorem~1.2.6 and Corollary~1.3.4 in \cite{Hormander1}):

\begin{lemma}[Whitney/Borel extension]\label{L:WB}
Given $k\in\{0,1,\ldots,\infty\}$ and functions $\rho_j(y)\in C^{k-j}_y(\R)$ for $0\leq j < k+1$, there exists a function $\rho\in C^{k}_{s,y}(\R^2)$ obeying
$$
[\partial^j_s \rho](0,y) = \rho_j(y) \qquad \text{for each} \quad 0\leq j < k+1.
$$
Naturally, if $k=\infty$, then $k-j=k+1=\infty$.
\end{lemma}

We will use Lemma~\ref{L:WB} to construct the parametrix when $\frac 4p$ is not an integer.  The remaining case is
slightly more involved and will be addressed afterwards.

\begin{proposition}[Parametrix]\label{P:parametrix}
Let $p>0$ such that $\frac4p$ is not an integer.  Fix $2\leq k\in \N\cup\{\infty\}$ and $\lambda\in C^{k-1}_{s,y}([0,\infty)\times\R)$ such that $\lambda(0,y) =1$ for all $y\in \R$.
Then there exists a nowhere-vanishing parametrix $\tilde v\in C^{k-1}_{s,y}((0,\infty)\times\R)$ for the equation \eqref{modified nlw} whose failure to satisfy the equation can be described as follows:
setting
$$
\tE:=\tilde v_{ss}   - \tilde v_{yy}  - \tfrac{2(p+2)}{p^2} \lambda(s,y) |\tilde v|^p \tilde v,
$$
we have
\begin{align}\label{accuracy}
\partial_s^\alpha \partial_y^\beta \tE(s,y) = o\bigl( s^{-\frac2p+k-3-\alpha} \bigr) \quad \text{as} \quad s\to 0,
\end{align}
for all $y\in \R$ and for all $0\leq \alpha+ \beta<k$.  When $k=\infty$, this means that the left-hand side of \eqref{accuracy} vanishes to all orders as $s\to 0$.
Moreover, $\tilde v$ blows up on the line $\{s=0,\, y\in\R\}$ in the sense that
\begin{equation}\label{tv as s 0}
\lim_{s\searrow 0}s^{\frac2p} \tilde v(s,y) =1 \quad\text{for all} \quad y\in \R.
\end{equation}
\end{proposition}

\begin{proof}
We start by noting that if $\lambda\equiv 1$, then there is a well-known solution to \eqref{modified nlw} that blows up on the line $\{s=0,y\in\R\}$, namely, $v(s,y)=s^{-2/p}$.  Therefore, we will look for a parametrix of the form $\tilde v(s,y) = s^{-2/p} \rho(s,y)$ with $\rho(0,y)\equiv1$.  The condition that $\tilde v$ satisfies \eqref{modified nlw} to the level of accuracy described in \eqref{accuracy} translates to the following condition for $\rho$:
\begin{align}\label{accuracy rho}
\partial_s^\alpha \partial_y^\beta \mathcal E(s,y) = o( s^{k-3-\alpha}),
\end{align}
uniformly in $y\in \R$ and for all $0\leq \alpha+ \beta<k$, where
\begin{equation}\label{rho eqn}
\mathcal E :=\rho_{ss} - \rho_{yy} - \tfrac4p s^{-1}\rho_s + \tfrac{2(p+2)}{p^2} s^{-2} [\rho- \lambda |\rho|^p\rho] .
\end{equation}
This in turn corresponds to conditions on $\rho_j(y):= [\partial^j_s \rho](0,y)$ for $0\leq j< k$.

We first seek to satisfy \eqref{accuracy rho} in the case $\alpha=\beta=0$; that \eqref{accuracy rho} continues to hold for all $0\leq \alpha+ \beta<k$ will be an immediate
consequence of our construction.  As $\rho(0,y)\equiv1$, we must have $\rho_0(y)\equiv 1$, which ensures cancellation of the $O(s^{-2})$ term on the left-hand side of
\eqref{accuracy rho} when $\alpha=\beta=0$.  To continue, we expand $\lambda$ in Taylor series near $s=0$ to maximal order,
$$
\lambda(s,y) = 1 +s \partial_s\lambda(0,y) + \cdots+ \frac{s^{k-1}}{(k-1)!} \partial_s^{(k-1)}\lambda(0,y) + o(s^{k-1}),
$$
and we substitute this into \eqref{rho eqn}.  With a little computation, one finds the following coefficients for
$s^{j}$ on the left-hand side of \eqref{accuracy rho} (in the case $\alpha=\beta=0$):
\begin{equation*}
\begin{cases}
0 &\quad \text{if } j=-2\\
-\tfrac{2(p+4)}p \rho_1(y) -\tfrac{2(p+2)}p \partial_s \lambda(0,y) &\quad \text{if } j=-1\\
\frac{j+3}{(j+2)!}\bigl(j-\tfrac4p\bigr)\rho_{j+2}(y) - \frac1{j!}\partial_y^2 \rho_j(y) + P_j &\quad \text{if } j\geq 0,
\end{cases}
\end{equation*}
where $P_j$ is a polynomial in $\partial_s\lambda(0,y), \ldots, \partial_s^{j+2}\lambda(0,y)$ and $\rho_1(y), \ldots, \rho_{j+1}(y)$.

To zero the coefficient of $s^{-1}$, we simply choose
$$
\rho_1(y):=-\tfrac{p+2}{p+4} \partial_s \lambda(0,y) \in C^{k-2}_y.
$$

More generally, since $4/p$ is not an integer, the factor in front of $\rho_{j+2}(y)$ is non-zero and so we may recursively choose this function to annihilate the
$O(s^j)$ term.  Simple induction shows that $\rho_j \in C^{k-j-1}$.  The continuation of this procedure is limited only by the regularity of $\lambda$.  More precisely,
when $k$ is finite, the process terminates with the choice of $\rho_{k-1}\in C^0(\R)$ to eliminate the $O(s^{k-3})$ term on the right-hand side of \eqref{rho eqn}.

By Lemma~\ref{L:WB}, there exists $\rho\in C^{k-1}_{s,y} ([0, \infty)\times\R)$ which matches our Taylor expansion at $s=0$.  This guarantees that \eqref{accuracy rho} holds,
which in turn shows that the parametrix $\tilde v\in C^{k-1}_{s,y} ((0, \infty)\times\R)$ satisfies \eqref{accuracy} and \eqref{tv as s 0}.

To see that the parametrix $\tilde v$ can be chosen to be nowhere-vanishing, we note that as $\rho_0(y)\equiv 1$, the extension $\rho(s,y)$ can be chosen to be nowhere-vanishing, indeed, bounded from below by a positive value.  This can been seen quite directly from the proof of Lemma~\ref{L:WB} or by post facto modification of any other extension.

This completes the proof of the proposition.
\end{proof}

The proof above breaks down in the case $4/p$ is an integer: it is impossible to cancel the $O(s^{4/p})$ term in
$\mathcal E$ because the coefficient of $\rho_{2+4/p}$ vanishes.  This phenomenon is well-known in the context of
power-series solutions to linear ODEs and our treatment of this case is adapted from that model.  For simplicity of exposition,
we will only consider the case where $\lambda$ is $C^\infty_{s,y}$.

To construct the parametrix in the case when $\frac4p$ is an integer (and the conformal factor is smooth), we will rely
on the following analogue of the Borel extension lemma, whose proof is easily adapted from that of \cite[Theorem~1.2.6]{Hormander1} cited earlier:

\begin{lemma}[Borel extension]\label{L:Bex}
Given $C^\infty_y(\R)$ functions $\rho_{j,\ell}(y)$ for integers $0\leq j<\infty$ and $0\leq \ell\leq L(j)<\infty$, there is a function $\rho\in C^{\infty}_{s,y}((0,\infty)\times\R)$ obeying
$$
\partial_s^\alpha \partial_y^\beta \Bigl[ \rho(s,y) - \sum_{j=0}^J \sum_{\ell=0}^{L(j)} \frac{\rho_{j,\ell}(y)}{j!} s^{j} \log^\ell(s) \Bigr] = O( s^{J+1-\alpha-\eps} ) \quad\text{as}\quad s\searrow 0
$$
for any $0\leq \alpha,\beta,J < \infty$ and any $\eps>0$.
\end{lemma}

\begin{proposition}[Parametrix]\label{P:parametrix'}
Suppose $\frac4p\in\N$ and $\lambda\in C^{\infty}_{s,y}([0,\infty)\times\R)$ with $\lambda(0,y)\equiv 1$. Then there is
a nowhere-vanishing parametrix $\tilde v\in C^{\infty}_{s,y}((0,\infty)\times\R)$ for the equation \eqref{modified nlw}
such that \eqref{tv as s 0} holds and
$$
\tE:=\tilde v_{ss}   - \tilde v_{yy}  - \tfrac{2(p+2)}{p^2} \lambda(s,y) |\tilde v|^p \tilde v
$$
along with all of its derivatives vanish to infinite order as $s\to0$.
\end{proposition}

\begin{proof}
Consider what happens when we substitute
$$
\rho = \sum_{j=0}^{1+\frac4p} \frac{\rho_{j,0}(y)}{j!} s^j + \frac{\rho_{2+\frac4p,1}(y)}{(2+\frac4p)!} s^{2+\frac4p} \log(s)
$$
into \eqref{rho eqn}.  To make the coefficients of $s^\ell$, for $-2\leq \ell<\frac4p$, in the resulting expansion vanish, we choose $\rho_{j,0}$ to be the functions $\rho_j$ from
Proposition~\ref{P:parametrix} for all $0\leq j\leq 1+\frac4p$.  Moreover, the coefficient of $s^{4/p}\log(s)$ vanishes for the same reason that the coefficient of $\rho_{2+4/p}$ vanished in the proof of Proposition~\ref{P:parametrix}.  It is possible to choose $\rho_{2+4/p,1}$ to cancel the coefficient of $s^{4/p}$, because this coefficient is now
$$
\frac{(3+\frac4p)\rho_{2+\frac4p,1}(y)}{(2+\frac4p)!} - \frac{\partial_y^2 \rho_{\frac4p,0}(y)}{(\frac4p)!} + P_{\frac4p},
$$
where, just as in the preceding proof, $P_{4/p}$ is a polynomial in $\rho_1(y), \ldots, \rho_{1+4/p}(y)$, as well as the first $2+4/p$ derivatives of $\lambda$.

To continue, we recursively choose coefficients in the expansion of the type given in Lemma~\ref{L:Bex} to construct our full parametrix.  The success of this procedure
relies on the fact that for $j>4/p$, one can choose $\rho_{j+2,\ell}$ to cancel the term of order $s^{j} \log^\ell(s)$ in the expansion of $\mathcal E$, where $\mathcal E$ is as in \eqref{rho eqn}.  To see that this is indeed possible, we note that the coefficient at this order is
$$
\frac{j+3}{(j+2)!}\bigl(j-\tfrac4p\bigr)\rho_{j+2,\ell}(y) - \frac1{j!}\partial_y^2 \rho_{j,\ell}(y) + P_{j,\ell}
$$
where $P_{j,\ell}$ is a polynomial in the Taylor coefficients of $\lambda$ and lower-order coefficients in the expansion of $\rho$.  By lower-order coefficients we mean
$\rho_{j',\ell'}$ with $j'<j$ or $j=j'$ and $\ell'<\ell$.  Note also that for our application, $L(j)=0$ for $0\leq j\leq 1+\frac4p$ and more generally, $L(j)=O(j)$ for any $0\leq j<\infty$.

That $\rho(s,y)$, and hence the parametrix $\tilde v$, can be chosen to be nowhere-vanishing follows from the fact that $\rho_{0,0}(y)\equiv 1$ and $L(0)=0$.  This completes the proof of the proposition.
\end{proof}

%
%
%
%

\section{Basic estimates for a singular wave equation} \label{S:estimates}

In this section, we prove finite speed of propagation and basic estimates for solutions to the critically singular wave equation
\begin{align}\label{singular eq}
v_{ss} - v_{yy} - (\nu^2-\tfrac14) s^{-2} v = F.
\end{align}
In our applications, $\nu=3/2 + 2/p$.  The reason behind the appearance of the singular perturbation will become apparent in Section~\ref{S:construction}; see \eqref{w eq}.

\begin{proposition}[Estimates for \eqref{singular eq}]\label{P:dispersive}
Fix $\nu\geq\frac12$ and let $k(s,y;s_0)$ be the causal fundamental solution to \eqref{singular eq}, that is,
\begin{align}\label{k eq}
k_{ss} - k_{yy} - (\nu^2-\tfrac14) s^{-2} k = \delta(s-s_0) \delta(y) \quad \text{for } s,s_0\in (0, \infty) \text{ and } y\in \R
\end{align}
with $k(s,y;s_0)\equiv 0$ for $s<s_0$.  Then $k$ obeys finite speed of propagation, namely,
\begin{align}\label{fsp}
k(s,y;s_0)\equiv 0 \quad \text{when} \quad  |y|>s-s_0,
\end{align}
and satisfies
\begin{align}\label{k bound 0}
\bigl\|k(s,y;s_0)\bigr\|_{L^1_y(\R)} \lesssim (s-s_0) \Bigl(\frac{s}{s_0}\Bigr)^{\nu-\frac12}.
\end{align}
Moreover, $\partial_{s_0} k$ and $\partial_{s} k$ are finite measures with total variation
\begin{align}\label{k bound 1}
\bigl\|\partial_{s_0} k(s,y;s_0)\bigr\|_{M_y} + \bigl\|\partial_{s} k(s,y;s_0)\bigr\|_{M_y}&\lesssim 1 + \frac{s-s_0}{s_0} \Bigl(\frac{s}{s_0}\Bigr)^{\nu-\frac12},
\end{align}
and, when $\nu\geq \frac32$,
\begin{align}\label{k bound 2}
\bigl\|\partial_{s} k(s,y;s_0)\bigr\|_{M_y} &\lesssim 1 + \frac{s-s_0}{s} \Bigl(\frac{s}{s_0}\Bigr)^{\nu-\frac12}.
\end{align}
\end{proposition}

When $\nu=1/2$, $k$ is the fundamental solution for the regular wave equation, namely,
$$
k_0(s,y;s_0) = \frac12 \chi_{\{|y|\leq s-s_0\}} \quad\text{or}\quad \hat k_0(s,\xi;s_0) = \frac{\sin(2\pi\xi[s-s_0])}{2\pi\xi}.
$$
The properties \eqref{fsp}, \eqref{k bound 0}, and \eqref{k bound 1} are obvious in this case.

Just as for the usual wave equation, the fundamental solution $k$ carries all the information necessary for the Duhamel formula:
\begin{equation}\label{duhamel}
v(s) = k(s;s_0) * v_s(s_0) - \tfrac{\partial k}{\partial s_0}(s;s_0) * v(s_0) + \int_{s_0}^s k(s;s_1) * F(s_1) \,ds_1,
\end{equation}
whenever $0<s_0<s$ and $v$ solves \eqref{singular eq}. Here $*$ represents convolution in the $y$ variable.  The estimates appearing in Proposition~\ref{P:dispersive}
allow us to address the possibility of sending $s_0 \downarrow 0$ in this formula, which we will do in the proof of Corollary~\ref{C:singular wave}.

\begin{proof}
Fix $\nu\geq \frac12$.  Claim \eqref{fsp} is inherited from the usual wave equation, as can be seen by iterated Duhamel expansion.
We will also give an analytical proof of this in what follows.

As equation \eqref{singular eq} is invariant under spatial translation, one can compute $k$ via its spatial Fourier transform
$$
\hat k (s,\xi;s_0):= \int_{\R} e^{-2\pi i y\xi} k(s,y;s_0)\, dy.
$$
Taking the spatial Fourier transform in \eqref{k eq} we obtain
\begin{align*}
\hat k_{ss} + 4\pi^2 \xi^2 \hat k - (\nu^2-\tfrac 14) s^{-2} \hat k = \delta (s-s_0).
\end{align*}
This is a version of the Bessel equation and has solution (cf. \cite[\S9.1.49]{AS})
\begin{equation*}
\hat k(s,\xi;s_0)= \begin{cases}
\tfrac{\pi}2 \sqrt{ss_0\rstrut}\bigl[ J_\nu(2\pi |\xi| s_0) Y_\nu(2\pi |\xi| s) - Y_\nu(2\pi |\xi| s_0) J_\nu(2\pi |\xi| s) \bigr], &\quad s>s_0\\
0, &\quad s\leq s_0,
\end{cases}
\end{equation*}
with $J_\nu$ and $Y_\nu$ denoting the Bessel functions of the first and second kind, respectively.  Notice that $\hat k$ is a continuous function in $s$, while $\partial_s \hat k$
has a jump discontinuity at $s=s_0$; indeed, $\partial_s \hat k \to 1$ as $s\downarrow s_0$, as follows from the Wronskian relation for Bessel functions (cf. \cite[\S9.1.16]{AS}).

For fixed $s,s_0\in (0,\infty)$, $\hat k$ extends from $\xi>0$ to an entire function in $\C$ which is even under the transformation $\xi\mapsto -\xi$.
Indeed, while both $J_\nu$ and $Y_\nu$ typically have branch singularities at the origin, these cancel out in the combination that forms $\hat k$; see \cite[\S9.1.2, 9.1.10, 9.1.11]{AS}.

One can also expand $\hat k$ in terms of the Hankel functions $H^{(1)}_\nu = J_\nu + iY_\nu$ and $H^{(2)}_\nu = J_\nu - iY_\nu$:
\begin{equation*}
\hat k(s,\xi;s_0)= \begin{cases}
\tfrac{\pi}{4i} \sqrt{ss_0\rstrut}\bigl[ H_\nu^{(2)}(2\pi \xi s_0) H_\nu^{(1)}(2\pi \xi s) - H_\nu^{(1)}(2\pi \xi s_0) H_\nu^{(2)}(2\pi \xi s)  \bigr] ,\!\!&  s>s_0\\
0, \!\!& s\leq s_0.
\end{cases}
\end{equation*}
Using the behavior of the Hankel functions in the whole complex plane (cf. \cite[\S9.2.3, 9.2.4]{AS}), we find that
$$
\bigl| \hat k(s,\xi;s_0) \bigr| = O_{s,s_0}\bigl (e^{2\pi\lvert\Im \xi| |s-s_0|} \bigr) \quad \text{as}\quad |\xi|\to \infty.
$$
Thus, by deforming the contour in the imaginary direction (up if $y>0$, down if $y<0$), one can deduce
$$
k(s,y;s_0)=\int_{\R} e^{2\pi i y\xi} \hat k(s,\xi;s_0)\, d\xi = 0 \quad \text{for} \quad |y|>s-s_0>0.
$$
This completes the proof of \eqref{fsp}.

It remains to prove \eqref{k bound 0}, \eqref{k bound 1}, and \eqref{k bound 2}.  As the kernel $k_0$ associated with the usual wave equation obeys the bound \eqref{k bound 0}, we need only control on the difference $k-k_0$ in $L^1_y$.  Both kernels also obey \eqref{fsp} and so
\begin{equation*}\label{E:un diff step}
\bigl\|k-k_0\bigr\|_{L^1_y (\R)} \leq 2|s-s_0| \bigl\|k-k_0\bigr\|_{L^\infty_y (\R)} \leq 2|s-s_0| \bigl\|\hat k-\hat k_0\bigr\|_{L^1_\xi (\R)}.
\end{equation*}
This is the strategy we use to prove \eqref{k bound 0}; compare \eqref{E:un diff goal} below.  We will prove \eqref{k bound 1} in an analogous manner; however, a more careful subtraction needs to be made.  Specifically, we will show that
\begin{equation}\label{E:diff goal}
\bigl\|\partial_{s_0} \hat k(s,\xi;s_0) - \partial_{s_0} \hat k_0(s,\xi;s_0) + \tfrac{4\nu^2-1}{8} \tfrac{s-s_0}{ss_0}\hat k_0(s,\xi;s_0) \bigr\|_{L^1_\xi (\R)}
    \lesssim s_0^{-1} \Bigl(\frac{s}{s_0}\Bigr)^{\nu-\frac12}.
\end{equation}
The analogue of \eqref{E:diff goal} for $\partial_s k$ is
\begin{equation}\label{Dsk}
\bigl\|\partial_{s} \hat k(s,\xi;s_0) - \partial_{s} \hat k_0(s,\xi;s_0) - \tfrac{4\nu^2-1}{8} \tfrac{s-s_0}{ss_0}\hat k_0(s,\xi;s_0) \bigr\|_{L^1_\xi (\R)}
    \lesssim s^{-1} \Bigl(\frac{s}{s_0}\Bigr)^{\nu-\frac12}.
\end{equation}
We will not give details of the proof of \eqref{Dsk}, because it differs very little from that of \eqref{E:diff goal}.  Note that \eqref{k bound 2} follows from \eqref{Dsk}; the necessity of taking
$\nu\geq \frac32$ comes from estimating the contribution of $\hat k_0$.

To complete the proof of \eqref{k bound 0} and \eqref{k bound 1}, we need only crude bounds on the Bessel functions for small values of $\xi$, namely,
\begin{equation}\label{E:JY small}
|J_\nu(z)| \lesssim \min\bigl(z^{\nu},z^{-1/2} \bigr) \quad\text{and}\quad |Y_\nu(z)| \lesssim \max\bigl(z^{-\nu},z^{-1/2} \bigr)
\end{equation}
for each fixed $\nu\geq\frac12$ and uniformly in $z\in(0,\infty)$.  For larger $\xi$ we use
\begin{equation} \label{J+Y}
 J_\nu(z) + i Y_\nu(z) = H^{(1)}_\nu(z) = \bigl( \tfrac2{\pi z} \bigr)^{1/2} a_\nu(z) e^{iz - i\pi(2\nu+1)/4},
\end{equation}
where $z\in[1,\infty)$ and $a_\nu$ is complex-valued and obeys (cf. \cite[\S9.2.5--10]{AS})
$$
a_\nu(z) = 1 + i\tfrac{4\nu^2-1}{8z} + O(z^{-2}) = 1 + O(z^{-1}) \quad \text{uniformly for $z\in[1,\infty)$.}
$$
Remember that $J_\nu$ and $Y_\nu$ are real-valued for such $z$.

To control $\partial_{s_0} k$, we combine the above with the relations
$$
\partial_z J_\nu(z) =\tfrac{\nu}z J_\nu(z) -  J_{\nu+1} (z) \quad \text{and} \quad \partial_z Y_\nu(z) =\tfrac{\nu}z Y_\nu(z) -  Y_{\nu+1} (z);
$$
see \cite[\S9.1.27]{AS}.  This yields
\begin{align*}
\partial_{s_0} \hat k(s,\xi; s_0)
&= \tfrac{\pi}2 \sqrt{ss_0\vphantom{i}}\Bigl\{\tfrac{2\nu+1}{2s_0}\bigl[J_\nu(2\pi |\xi| s_0) Y_\nu(2\pi |\xi| s) -  Y_\nu(2\pi |\xi| s_0) J_\nu(2\pi |\xi| s)\bigr] \\
&\qquad\quad -2\pi|\xi|\bigl[ J_{\nu+1}(2\pi |\xi| s_0) Y_\nu(2\pi |\xi| s) - Y_{\nu+1}(2\pi |\xi| s_0) J_\nu(2\pi |\xi| s)  \bigr]\Bigr\},
\end{align*}
whenever $s>s_0$.

We now turn to the details of showing \eqref{E:diff goal} and
\begin{equation}\label{E:un diff goal}
\bigl\| \hat k(s,\xi;s_0) - \hat k_0(s,\xi;s_0) \bigr\|_{L^1_\xi(\R)} \lesssim \Bigl(\frac{s}{s_0}\Bigr)^{\nu-\frac12}.
\end{equation}
As some terms can be treated simultaneously, we use $\alpha\in\{ 0,1\}$ to indicate the number of $s_0$ derivatives.  We divide the integrals into three regions, as follows:

\textbf{Region I:} $|\xi|\leq s^{-1}$.  For $\xi$ in this region, $\alpha\in \{0,1\}$, and $s>s_0$, we have
\begin{align*}
\bigl|\partial_{s_0}^\alpha \hat k(s,\xi;s_0) \bigr| + \bigl|\partial_{s_0}^\alpha \hat k_0(s,\xi;s_0) \bigr| \lesssim  \sqrt{ss_0\rstrut} \Bigl(\frac{s}{s_0}\Bigr)^{\nu} s_0^{-\alpha} ,
\end{align*}
as follows easily from \eqref{E:JY small}.
Thus, we can estimate the contributions of this region by
\begin{align}\label{r1}
\int_{|\xi|\leq s^{-1}} \bigl|\partial_{s_0}^\alpha \hat k(s,\xi;s_0) \bigr| + \bigl|\partial_{s_0}^\alpha \hat k_0(s,\xi;s_0) \bigr| \, d\xi \lesssim s_0^{-\alpha}  \Bigl(\frac{s}{s_0}\Bigr)^{\nu-\frac12}.
\end{align}

\textbf{Region II:} $s^{-1} \leq |\xi| \leq s_0^{-1}$. Using \eqref{E:JY small} again, we find
\begin{align*}
\bigl|\partial_{s_0}^\alpha \hat k(s,\xi;s_0) \bigr| + \bigl|\partial_{s_0}^\alpha \hat k_0(s,\xi;s_0) \bigr| \lesssim  \sqrt{ss_0\rstrut} |\xi|^{-\nu-\frac12} s_0^{-\nu-\alpha} s^{-\frac12}.
\end{align*}
Therefore, we bound the contribution of this region by
\begin{align}\label{r2}
\int_{s^{-1} \leq |\xi| \leq s_0^{-1}} \bigl|\partial_{s_0}^\alpha \hat k(s,\xi;s_0) \bigr| + \bigl|\partial_{s_0}^\alpha \hat k_0(s,\xi;s_0) \bigr| \, d\xi \lesssim s_0^{-\alpha}  \Bigl(\frac{s}{s_0}\Bigr)^{\nu-\frac12}.
\end{align}
In the trivial case $\nu=\frac12$, integrating the bound given above actually produces a factor of $\log(s/s_0)$ rather than $(s/s_0)^0=1$.

\textbf{Region III:} $|\xi|\geq s_0^{-1}$.  Now we use formula \eqref{J+Y}.  This quickly yields
\begin{align*}
\bigl|\hat k(s,\xi;s_0) - \hat k_0(s,\xi;s_0) \bigr| \lesssim  \frac{1}{s_0|\xi|^2}
\end{align*}
and consequently,
\begin{align}\label{r3a}
\int_{|\xi|>s_0^{-1}} \bigl|\hat k(s,\xi;s_0) - \hat k_0(s,\xi;s_0) \bigr| \, d\xi  \lesssim 1 \lesssim \Bigl(\frac{s}{s_0}\Bigr)^{\nu-\frac12} .
\end{align}

For the derivative, the analogue of the first calculation is longer; it yields
\begin{align*}
\bigl|\partial_{s_0} \hat k(s,\xi;s_0) - \partial_{s_0} \hat k_0(s,\xi;s_0) + \tfrac{4\nu^2-1}{8} \tfrac{s-s_0}{ss_0}\hat k_0(s,\xi;s_0) \bigr| \lesssim  \frac{1}{s_0^2|\xi|^2}.
\end{align*}
(The $O(|\xi|^{-1})$ terms collapse so elegantly because $\tfrac{4(\nu+1)^2-1}{8s_0} - \tfrac{2\nu+1}{2s_0} = \tfrac{4\nu^2-1}{8s_0}$.)  This leads immediately to
\begin{align}\label{r3b}
\int_{|\xi|>s_0^{-1}} \bigl|\partial_{s_0} \hat k(s,\xi;s_0) - \partial_{s_0} \hat k_0(s,\xi;s_0) + \tfrac{4\nu^2-1}{8} \tfrac{s-s_0}{ss_0}\hat k_0(s,\xi;s_0) \bigr| \, d\xi  \lesssim s_0^{-1},
\end{align}
which is consistent with \eqref{E:diff goal}.

Combining \eqref{r1}, \eqref{r2}, \eqref{r3a}, and \eqref{r3b} proves \eqref{E:diff goal} and \eqref{E:un diff goal}, which in turn complete the proof of the proposition.
\end{proof}

\begin{corollary}[Estimates on solutions to the singular wave equation]\label{C:singular wave}
Fix $\nu>\frac32$, $\delta>\nu-\tfrac32$, and a compact interval $0\in I\subseteq [0, \infty)$. Consider the initial-value problem
\begin{equation*}
\begin{cases}
v_{ss} - v_{yy} - (\nu^2-\tfrac14) s^{-2} v = F\\
v(0)=0, \quad v_s(0)=0
\end{cases}
\end{equation*}
for $v:I\times\R\to \R$.  If $F\in C^0_{s,y}(I\times\R)$ with $s^{-\delta} F(s,y) \in L^\infty_{s,y}(I\times\R)$, then this admits a distributional solution with $v$ and $v_s$ continuous; moreover, for integers $\alpha, \beta\geq 0$ and $s\in I$,
\begin{align}\label{v diff}
 \bigl\| \partial_s^\alpha \partial_y^\beta v(s,y)\bigr\|_{L_{y}^\infty(\R)}
\lesssim s^{2+\delta} \sum_{\mu=0}^\alpha \bigl\|s^{\mu-\alpha-\delta} \partial_s^{\mu} \partial_y^\beta  F\bigr\|_{L_{s,y}^\infty(I\times\R)},
\end{align}
 whenever the right-hand side is finite.
\end{corollary}

\begin{proof}
If $F^{(\eps)}\in C^\infty_{s,y}$ and supported on $[\eps,\infty)\times\R$, then the standard theory of strictly hyperbolic equations applies (cf. \cite[Ch. XXIII]{Hormander3}) and
yields a solution $v^{(\eps)} \in C^\infty_{s,y}$ that vanishes when $s<\eps$ and also obeys
\begin{gather*}
\bigl[\partial_{s}^2 - \partial_{y}^2 - (\nu^2-\tfrac14) s^{-2} \bigr] \partial_s^\alpha\partial_y^\beta v^{(\eps)} = \partial_y^\beta F_{\alpha}^{(\eps)} \\
\text{where} \qquad    F_{\alpha}^{(\eps)} : = \partial_s^\alpha F^{(\eps)} +
    (\nu^2-\tfrac14) \sum_{\mu=0}^{\alpha-1} \tbinom\alpha\mu \bigl[\partial_s^{\alpha-\mu} s^{-2}\bigr]\bigl[\partial_s^{\mu}v^{(\eps)} \bigr].
\end{gather*}
The key issue is to obtain bounds independent of $\eps$ and convergence as $\eps\downarrow0$ without relying on the additional smoothness of $F^{(\eps)}$.

For $0<\eps<s\leq \eps'<\infty$,
\begin{align*}
\bigl| \partial_s^\alpha\partial_y^\beta \bigl[v^{(\eps)} - v^{(\eps')}  \bigr] (s,y)\bigr| &= \bigl| \partial_s^\alpha\partial_y^\beta v^{(\eps)} (s,y)\bigr|\\
&\leq \biggl| \int_0^s  \int_\R k(s,z;s_0) \partial_y^\beta F^{(\eps)}_\alpha(s_0,y-z)\,dz\,ds_0 \biggr|
\end{align*}
and so by \eqref{k bound 0} we see that for $\delta>\nu-\frac32$,
\begin{align*}
\bigl| \partial_s^\alpha\partial_y^\beta v^{(\eps)} (s)\bigr|
    &\lesssim_{\alpha,\delta}  (s-\eps)^2 s^{\delta} \Bigl\{ \bigl\| s^{-\delta} \partial_s^\alpha \partial_y^\beta  F^{(\eps)} \bigr\|_{L_{s,y}^\infty}  +
        \sum_{\mu=0}^{\alpha-1} \bigl\| s^{\mu-\alpha-\delta-2} \partial_s^{\mu} \partial_y^\beta v^{(\eps)} \bigr\|_{L_{s,y}^\infty} \Bigr\}.
\end{align*}
Combining this with induction in $\alpha$ yields
\begin{align*}
\bigl| \partial_s^\alpha\partial_y^\beta \bigl[v^{(\eps)} - v^{(\eps')}  \bigr] (s,y)\bigr| &= \bigl| \partial_s^\alpha\partial_y^\beta v^{(\eps)} (s,y)\bigr|\\
    &\lesssim_{\alpha,\delta}  (s-\eps)^2 s^{\delta} \sum_{\mu=0}^{\alpha}  \bigl\| s^{\mu-\alpha-\delta} \partial_s^\mu \partial_y^\beta  F^{(\eps)} \bigr\|_{L_{s,y}^\infty}
    \end{align*}
uniformly for $0<\eps<s\leq \eps'<\infty$ and $y\in\R$.  In addition to showing convergence, these bounds also provide uniform control by sending $\eps\downarrow0$ and
$\eps'\to\infty$.

By using \eqref{k bound 2} in addition to the above, it is possible to control one more derivative of $v^{(\eps)}$:
\begin{align*}
\bigl| \partial_s^{\alpha+1} \partial_y^\beta v^{(\eps)}(s,y) \bigr|  &\leq \biggl| \int_0^s \int_\R k_s(s,z;s_0) \partial_y^\beta F^{(\eps)}_\alpha(s_0,y-z)\,dz\,ds_0 \biggr| \\
&\lesssim (s-\eps) s^{\delta} \sum_{\mu=0}^{\alpha}  \bigl\| s^{\mu-\alpha-\delta} \partial_s^\mu \partial_y^\beta  F^{(\eps)} \bigr\|_{L_{s,y}^\infty}.
\end{align*}
Notice that the additional derivative costs one power of $s$ on the right-hand side.

It remains to estimate $v^{(\eps)} - v^{(\eps')}$ in the region $0<\eps<\eps'<s<\infty$.  By the Duhamel formula \eqref{duhamel},
\begin{align*}
\bigl| \partial_s^\alpha\partial_y^\beta \bigl[v^{(\eps)} - v^{(\eps')}  \bigr](s) \bigr| &\leq  \bigl| k(s;\eps') * \partial_s^{\alpha+1}\partial_y^\beta  v^{(\eps)}(\eps') \bigr|
    + \bigl| \tfrac{\partial k}{\partial s_0}(s;\eps') * \partial_s^\alpha\partial_y^\beta  v^{(\eps)}(\eps') \bigr| \\
& \qquad + \int_{\eps'}^s \bigl| k(s;s_0) * \partial_y^\beta [ F_{\alpha}^{(\eps)}(s_0) - F_{\alpha}^{(\eps')}(s_0)]\bigr| \,ds_0.
\end{align*}
Combining \eqref{k bound 0} and \eqref{k bound 1}, with
$$
\bigl| \partial_s^\alpha\partial_y^\beta  v^{(\eps)}(\eps') \bigr| \lesssim_F (\eps'-\eps)^2 (\eps')^{\delta} \quad\text{and}\quad
    \bigl| \partial_s^{\alpha+1}\partial_y^\beta  v^{(\eps)}(\eps') \bigr| \lesssim_F (\eps'-\eps) (\eps')^{\delta},
$$
which where proved previously, we see that the first two terms obey
\begin{align*}
\bigl| k(s;\eps') * \partial_s^{\alpha+1}\partial_y^\beta  v^{(\eps)}(\eps') \bigr|
    + \bigl| \tfrac{\partial k}{\partial s_0}(s;\eps') * \partial_s^\alpha\partial_y^\beta  v^{(\eps)}(\eps') \bigr|
    \lesssim_F s^{\nu+\frac12} (\eps')^{\delta-\nu+\frac12} (\eps'-\eps),
\end{align*}
and hence converge to zero as $\eps'\to0$, since $\delta>\nu-\frac32$.

To estimate the contribution of the inhomogeneity, we argue exactly as in the case $\eps<s\leq \eps'$ to obtain
\begin{align*}
\int_{\eps'}^s \!\bigl| k(s;s_0) * \partial_y^\beta [ F_{\alpha}^{(\eps)}(s_0) \!- \! F_{\alpha}^{(\eps')}(s_0)]\bigr| \,ds_0
    \lesssim s^{2+\tilde\delta}\! \sum_{\mu=0}^{\alpha}  \bigl\| s^{\mu-\alpha-\tilde\delta} \partial_s^\mu \partial_y^\beta  [ F^{(\eps)} \!-\! F^{(\eps')}] \bigr\|_{L_{s,y}^\infty},
\end{align*}
where $\nu-\frac32<\tilde\delta<\delta$ and $L_{s,y}^\infty$ relates to the slab $[\eps',s]\times\R$.   Using the fact that $\tilde\delta<\delta$ when $s$ is small
and that $F^{(\eps)}\to F$ to cover larger values of $s$, we see that the $L^\infty$ norm above converges to zero as $\eps'\to 0$.

This completes the proof of the convergence and thus settles the lemma.
\end{proof}

%
%
%
%

\section{Construction of the blowup solution} \label{S:construction}

In this section we prove Theorem~\ref{T:main} and Corollary~\ref{C:Cantor}.  Theorem~\ref{T:main} will follow easily from Corollary~\ref{C:line reduction} and the following

\begin{theorem}[Regular solutions to \eqref{modified nlw} that blow up on a line]\label{T:main'}
Suppose either {\rm(i)} $p>0$ and $k=\infty$, or {\rm(ii)} $\frac4p\not\in \N$ and $3+\frac4p<k\in \N\cup\{\infty\}$.
Let $\lambda\in C^{k-1}_{s,y}([0,\infty)\times\R)$ such that $\lambda(0,y) =1$ for all $y\in \R$.
Then there exist a time $s_0(p,\lambda)>0$ and a solution $v\in C^{\lfloor k-3-4/p\rfloor}_{s,y}((0, s_0]\times\R)$ to \eqref{modified nlw} which blows up on the line
$\{s=0,\, y\in \R\}$ in the sense that
\begin{align}\label{blow}
\lim_{s\searrow 0}s^{\frac2p} v(s,y) =1 \quad\text{for all} \quad y\in \R.
\end{align}
Additionally, $s_0$ admits a positive lower bound that depends only on the $C^{\lceil 2+4/p \rceil}_{s,y}$ norm of $\lambda$.
\end{theorem}

\begin{proof}
Given $s_0,\delta > 0$ and an integer $m\geq 0$, consider the Banach space $X^m_\delta([0,s_0])$ consisting of those $w\in C^m_{s,y}([0,s_0]\times \R)$ for which the norm
$$
\| w \|_{X^m_\delta([0,s_0])} := \sum_{\alpha+\beta\leq m} \sum_{\mu\leq\alpha} \  \| s^{-(\alpha-\mu)-\delta-1} \partial_s^\mu \partial_y^\beta w \|_{L^\infty_{s,y}([0,s_0]\times \R)}
$$
is finite.  This is the family of spaces in which we will run a Picard iteration (= contraction mapping) argument to improve the parametrices $\tilde v$ from Propositions~\ref{P:parametrix}~and~\ref{P:parametrix'} to an actual solution $v=\tilde v +w$.  In terms of the spaces just introduced, the error in the parametrices, namely,
$$
\tE:=\tilde v_{ss}   - \tilde v_{yy}  - \tfrac{2(p+2)}{p^2} \lambda(s,y) |\tilde v|^p \tilde v,
$$
satisfies
\begin{align}\label{Et in X}
\| \tE \|_{X^m_{\delta-1}([0,s_0])} \lesssim 1 \quad\text{whenever}\quad m+\delta \leq k-3-\tfrac2p.
\end{align}
In fact, this quantity is $o(1)$ as $s_0\to 0$; however, we will not make use of this.

In order for $v=\tilde v + w$ to be a solution to \eqref{modified nlw}, we must have
\begin{equation*}
w_{ss}   - w_{yy} - \tfrac{2(p+2)}{p^2} \lambda(s,y) \bigl(|v|^pv- |\tilde v|^p \tilde v\bigr) + \tE = 0.
\end{equation*}
By construction, $\tilde v(s,y) = s^{-2/p}\rho(s,y)$ with $\rho\in C^{k-1}_{s,y}$ and $\rho(0,y)\equiv 1$.  Using this and Taylor's formula, the equation for $w$ reduces to
\begin{align}\label{w eq}
w_{ss} - w_{yy} - \bigl(\nu^2-\tfrac14\bigr) s^{-2} w &= F,
\end{align}
where $\nu= \tfrac32+\tfrac2p$,
\begin{align*}
F(s,y)&:= -\tE + \tfrac{2(p+2)(p+1)}{p^2} s^{-2} (\lambda\rho^p-1) w + \tfrac{2(p+2)}{p^2} \lambda  \!\int_0^1\! w^2 f''\bigl(\tilde v + \theta w\bigr) (1-\theta) \, d\theta,
\end{align*}
and $f(w)=|w|^p w$.  We will soon estimate $F$ in the space $X^m_{\delta-1}$.  The motivation for this (and indeed, the introduction of these norms themselves) is the fact that
Corollary~\ref{C:singular wave} provides the necessary estimates for solutions to \eqref{w eq}.  Specifically,
\begin{align}\label{X estimates}
\| w \|_{X^m_\delta([0,s_0])} \lesssim_m  s_0 \| F \|_{X^m_{\delta-1}([0,s_0])} \quad\text{provided}\quad \delta > \tfrac2p.
\end{align}
Note that to derive this, one must consider \eqref{v diff} with varying values for $\delta$, not merely that appearing in the formula above.

To continue, note that $(\lambda\rho^p)(0,y)\equiv1$ and so, with a little extra work, one finds
$$
 \bigl\| \partial_s^\mu \partial_y^\beta  s^{-2} (\lambda\rho^p-1) \bigr\|_{L^\infty_y} \lesssim s^{-1-\mu} \quad\text{when}\quad \mu+\beta < k-1.
$$
Therefore,
\begin{align}\label{dude1 in X}
\bigl\| s^{-2} (\lambda\rho^p-1) w \bigr\|_{X^m_{\delta-1}([0,s_0])}
\lesssim \| w \|_{X^m_\delta([0,s_0])}   \quad\text{provided}\quad m < k-1.
\end{align}

We now move to the third term in $F$, which involves $f(w)=|w|^p w$.  Note that $f$ fails to be $C^\infty$ at $w=0$. This will not be a problem for us, because for $s_0$ small
we have $\tilde v(s,y) \sim s^{-2/p}$ uniformly on $[0,s_0]\times\R$; moreover, from the argument that follows, we see that $\|w\|_{L^\infty_{s,y}([0,s_0]\times\R)} \ll 1$.  As a result of this,
$$
\bigl| f^{(\ell+2)}(\tilde v + \theta w) \bigl| \sim s^{-\frac2p(p-\ell-1)} \quad\text{on}\quad [0,s_0]\times\R
$$
uniformly for $\theta\in[0,1]$.  Using this we find
\begin{align*}
\Bigl| \partial_s^\alpha \partial_y^\beta f''(\tilde v + \theta w) \Bigr| \lesssim \sum_{\ell=0}^{\alpha+\beta} \sum_{\vec\alpha\,\vdash\alpha} \sum_{\vec\beta\,\vdash\beta}
    s^{-\frac2p(p-\ell-1)} \prod_{j=1}^\ell \Bigl[ s^{-\frac2p-\alpha_j} + \bigl|\partial_s^{\alpha_j}\partial_s^{\beta_j} w\bigr| \Bigr],
\end{align*}
where $\vec\alpha,\vec\beta$ are partitions of $\alpha,\beta$ into $\ell$ parts.  This directly yields
\begin{align*}
\Bigl| \partial_s^\alpha \partial_y^\beta f''(\tilde v + \theta w) \Bigr| \lesssim s^{-\frac2p(p-1)-\alpha} \exp\bigl(\|w\|_{X^m_\delta} \bigr)
    \quad\text{for } \alpha+\beta\leq m \text{ and any } \delta >0.
\end{align*}
Naturally, the implicit constant depends on $m$.  Proceeding from here we find
\begin{align}\label{dude2 in X}
\bigl\| \lambda w^2 f''\bigl(\tilde v + \theta w\bigr) \bigr\|_{X^m_{\delta-1}([0,s_0])}
\lesssim \| w \|_{X^m_\delta([0,s_0])}^2 \exp\bigl( \|w\|_{X^m_\delta([0,s_0])} \bigr)
\end{align}
uniformly for $\theta\in[0,1]$, provided $m < k$.  Indeed, one can exhibit additional (helpful) positive powers of $s_0$ on the right-hand side.

With minor modifications, these estimates also show that
\begin{equation}\label{dude2' in X}
\begin{aligned}
\bigl\| \lambda w^2 f''\bigl(\tilde v &+ \theta w\bigr) - \lambda \tilde w^2 f''\bigl(\tilde v + \theta \tilde w\bigr) \bigr\|_{X^m_{\delta-1}([0,s_0])} \\
&\lesssim \| w - \tilde w \|_{X^m_\delta([0,s_0])} \exp\bigl( 2\|w\|_{X^m_\delta([0,s_0])} + 2\|\tilde w\|_{X^m_\delta([0,s_0])} \bigr).
\end{aligned}
\end{equation}

Let us now put all the pieces together.  We choose $m\leq \lfloor k-3-\frac4p\rfloor$, finite when $k=\infty$, and $\delta > \frac2p$ so that $m+\delta\leq k-3-\frac2p$.
When $k<\infty$ this is possible, since in this case $4/p$ is not an integer.  By combining \eqref{Et in X}, \eqref{X estimates}, \eqref{dude1 in X}, and \eqref{dude2 in X}, we see that the sequence of Picard iterates $w_n$ obey
$$
\| w_{n+1} \|_{X^m_\delta([0,s_0])} \lesssim s_0 \Bigl\{ 1 + \| w_n \|_{X^m_\delta([0,s_0])} + \| w_n \|_{X^m_\delta([0,s_0])}^2 \exp\bigl( \|w_n\|_{X^m_\delta([0,s_0])} \bigr) \Bigr\}.
$$
Thus by choosing $s_0$ sufficiently small, we can guarantee that $w_n$ are uniformly small in $X^m_\delta([0,s_0])$.

The term estimated in \eqref{dude1 in X} is linear in $w$ and so the bound works equally well for differences.  Using this and \eqref{dude2' in X} we obtain
$$
\| w_{n+1} - w_n \|_{X^m_\delta} \lesssim s_0 \| w_n - w_{n-1} \|_{X^m_\delta} \exp\bigl(2 \|w_n\|_{X^m_\delta} +2 \|w_{n-1}\|_{X^m_\delta} \bigr),
$$
which shows that the Picard iterates do indeed converge (for $s_0$ small).

For each integer $m\leq \lfloor k-3 -\frac4p\rfloor$, we have proved the existence of a $C^{m}_{s,y}$ solution to \eqref{modified nlw} on a small neighborhood of $s=0$,
which additionally satisfies \eqref{blow}.  Naively, the interval $[0,s_0]$ of existence may depend on $m$.  However, simple (inductive) persistence of regularity arguments
using only Duhamel's formula for the one-dimensional wave equation show that the solution $v$ cannot blow up in $C^m_{s,y}$ while remaining bounded in $C^0_{s,y}$.
We use, of course, that $\lambda\in C^{k-1}_{s,y}$.  This observation has two important consequences:  First, in the case $k=\infty$ it shows that the solution we construct
is indeed $C^\infty$.  Second, it shows that the time of existence depends only on the  $C^{\lceil 2+4/p \rceil}_{s,y}$ norm of $\lambda$.
\end{proof}

The reader will have undoubtedly noticed that in order to produce any solution at all (say merely $C^0_{s,y}$), this argument requires a great deal of regularity on
the conformal factor $\lambda$. This is not purely an artifact of the proof, but reflects a fundamental underlying phenomenon, which we explain next.

Even in the case $\lambda\equiv 1$, there are multiple solutions that diverge at $s=0$ at exactly the rate prescribed in \eqref{blow}.  Indeed, there is a one-parameter family of
such solutions that depend on $s$ alone.  To be more precise, using energy conservation, one sees that any such solution $v(s)$ obeys
$$
s = \int_{v(s)}^\infty \tfrac{p}2 a^{-\frac{p+2}{2}} \Bigl[1+ \tfrac{p^2}{2} E a^{-(p+2)} \bigr]^{-\frac12} \, da,
$$
where $E=\frac12|v_s|^2 - \frac{2}{p^2}|v|^{p+2}$ denotes the (conserved) energy density.  From this it is easy to derive the asymptotic behavior
$$
v(s) = s^{-2/p} \Bigl( 1 -\tfrac{p^2 E}{2(3p+4)} s^{2(p+2)/p} + O\bigl(s^{4(p+2)/p}\bigr) \Bigr).
$$
Thus, to uniquely specify such a solution we require a parametrix $\tilde v$ for which
$$
\tilde v_{ss}   - \tilde v_{yy}  - \tfrac{2(p+2)}{p^2} |\tilde v|^p \tilde v = o(s^{2/p}),
$$
which corresponds to $\alpha=0$ and $k=\lceil 3+\frac4p \rceil$ in \eqref{accuracy}.  In this sense, our regularity requirements are minimal for the construction
of $w$ by contraction mapping arguments, indeed, by any method that ensures uniqueness.

We now return to the main theme of the section and prove the results stated in the introduction.

\begin{proof}[Proof of Theorem~\ref{T:main}]  As mentioned at the beginning of the section, the construction will be a direct application of Corollary~\ref{C:line reduction} and Theorem~\ref{T:main'}.  Suppose therefore that either {\rm(i)} $p>0$ and $k=\infty$, or {\rm(ii)} $\frac4p\not\in \N$ and $3+\frac4p<k\in \N\cup\{\infty\}$.  Let $\Sigma=\{t=\sigma(x)\}$ be a $C^k$ uniformly space-like hypersurface.  Then by Corollary~\ref{C:line reduction}, the solution $v$ to \eqref{modified nlw} described in Theorem~\ref{T:main'} gives rise to a solution $u\in C^{\lfloor k-3 -\frac4p\rfloor}_{t,x} $ to \eqref{nlw} in a neighborhood of the hypersurface, namely, on
$\{\sigma_0(x)\leq t<\sigma(x)\}$, where the Cauchy hypersurface $\Sigma_0=\{t=\sigma_0(x)\}$ is the image of the line $\{s=s_0(p,\lambda), \, y\in\R\}$ via the conformal mapping $\Phi$ from Proposition~\ref{P:conformal exists}. Here $s_0(p,\lambda)$ is as in Theorem~\ref{T:main'}.  Moreover, by Corollary~\ref{C:line reduction} and
\eqref{blow}, the solution $u$ has the desired blowup behavior.
\end{proof}

\begin{proof}[Proof of Corollary~\ref{C:Cantor}]
Let $E\subset\R$ be a compact set.  We start by constructing a smooth function that has $E$ as a level set.  Indeed, let $I_j$ denote the bounded connected components of
$\R\setminus E$ and let $\phi:\R\to[0,1]$ be a smooth function, with $\phi(x)\equiv 0$ when $x\leq 0$, $\phi(x)>0$ when $x>0$, and $\phi(x)\equiv 1$ when $x\geq1/2$.
We define the hypersurface $\Sigma=\{t=\sigma(x)\}$ as follows:
\begin{equation*}
\sigma(x)= \eps
\begin{cases}
1, &\quad x\in E\\
1+\exp\bigl(-\frac1{|I_j|}\bigr) \phi\bigl( \tfrac{\dist(x,\partial I_j)}{|I_j|}\bigr), &\quad x\in I_j\\
1+\phi\bigl(\dist(x,E)\bigr), &\quad x<\inf E \, \text{ or }\,  x>\sup E.
\end{cases}
\end{equation*}
Here $\eps>0$ is a small constant that will be chosen later. Observe that $\| \sigma \|_{C^k_x} \lesssim_k \eps$, for each integer $k$.

Now consider the conformal mapping $\Phi$ of $\{s\geq 0,y\in\R\}$ onto $\{t\leq\sigma(x)\}$ constructed in Proposition~\ref{P:conformal exists}.
From this construction, we see that for $\eps$ small,
$$
\bigl\| D\Phi - \Id\bigr\|_{C^k_{s,y}} \lesssim_k \eps
$$
for any integer $k$, where $D\Phi$ denotes the derivative of $\Phi$.  Thus,
$$
\dist( \Phi^{-1}(\{0,x\}), \{s=0\} \bigr) \lesssim \eps
$$
and the conformal factor obeys
$$
\| \lambda - 1 \|_{C^k_{s,y}} \lesssim_k \eps
$$
for each integer $k$.  Combining these two observations, we see that when $\eps$ is sufficiently small, the smooth solution to
$$
v_{ss}   - v_{yy}  - \tfrac{2(p+2)}{p^2} \lambda |v|^p v = 0
$$
that was constructed in Theorem~\ref{T:main'} will exist throughout the region $\Phi^{-1}(\{0\leq t < \sigma(x)\})$.  (Of course, this relies on the fact that $s_0(p,\lambda)$ depends robustly on $\lambda$.)  Therefore, $u=v\circ \Phi^{-1}$ is a smooth solution to \eqref{nlw} defined on $\{0\leq t < \sigma(x)\}$ that blows up at time $t=\eps$ precisely on the set $E$.
\end{proof}

%
%
%
%

\end{document}